\documentclass{article}
\usepackage{geometry}
\usepackage{graphicx}	
\usepackage[cp1251]{inputenc}
\usepackage{amsfonts,amssymb,mathrsfs,amscd,amsmath,amsthm}
\usepackage{verbatim}
\DeclareGraphicsExtensions{.pdf,.png,.jpg}
\def\thtext#1{
  \catcode`@=11
  \gdef\@thmcountersep{. #1}
  \catcode`@=12
}

\def\threst{
  \catcode`@=11
  \gdef\@thmcountersep{.}
  \catcode`@=12
}
\def\opt{{\operatorname{opt}}}
\def\ig#1#2#3#4{\begin{figure}[!ht]\begin{center}%
\includegraphics[height=#2\textheight]{ps//#1.ps}\caption{#4}\label{#3}%
\end{center}\end{figure}}

\theoremstyle{plain}
\newtheorem{thm}{Theorem}
\newtheorem{prop}{Proposition}
\newtheorem{cor}[prop]{Corollary}
\newtheorem{ass}[prop]{Assertion}
\newtheorem{lem}[prop]{Lemma}

\theoremstyle{definition}
\newtheorem{dfn}{Definition}
\newtheorem{rk}{Remark}
\newtheorem{constr}{Construction}
\newtheorem{examp}{Example}
\def\ig#1#2#3#4{\begin{figure}[!ht]\begin{center}%
\includegraphics[height=#2\textheight]{#1.eps}\caption{#4}\label{#3}%
\end{center}\end{figure}}

 \catcode`@=11
 \def\.{.\spacefactor\@m}
 \catcode`@=12

\newcommand{\D}{\Delta}
\newcommand{\e}{\varepsilon}
\newcommand{\g}{\gamma}

\renewcommand{\l}{\lambda}

\renewcommand{\r}{\rho}
\newcommand{\s}{\sigma}
\renewcommand{\t}{\tau}

\newcommand{\N}{\mathbb{N}}

\newcommand{\R}{\mathbb{R}}

\newcommand{\rom}[1]{{\em #1}}

\renewcommand{\)}{\rom)}

\renewcommand{\:}{\colon}
\newcommand{\0}{\emptyset}

\renewcommand{\c}{\circ}

\newcommand{\oPi}{\stackrel{\raise-2pt\hbox{$\c$}}\Pi}
\newcommand{\oW}{\stackrel{\raise-2pt\hbox{$\c$}}W}
\newcommand{\sm}{\setminus}

\renewcommand{\ss}{\subset}

\newcommand{\x}{\times}

\newcommand{\diam}{{\operatorname{diam}}\,}

\newcommand{\GH}{\operatorname{\mathcal{G\!H}}}
\renewcommand{\min}{{\operatorname{min}}\,}

\newcommand{\dis}{{\operatorname{dis}}\,}
\newcommand{\id}{{\operatorname{id}}}

\def\opt{{\operatorname{opt}}}

\newcommand{\cA}{\mathcal{A}}
\newcommand{\cB}{\mathcal{B}}
\newcommand{\cC}{\mathcal{C}}

\newcommand{\cH}{\mathcal{H}}
\newcommand{\cM}{\mathcal{M}}
\newcommand{\cN}{\mathcal{N}}
\newcommand{\cP}{\mathcal{P}}
\newcommand{\cR}{\mathcal{R}}

\makeatletter
\def\blfootnote{\xdef\@thefnmark{}\@footnotetext}
\makeatother

\begin{document}

\title{Isometric Embeddings of Bounded Metric Spaces into the Gromov--Hausdorff Class}
\author{Alexander O.~Ivanov, Alexey A.~Tuzhilin}
\date{}
\maketitle

\begin{abstract}
It is shown that any bounded metric space can be isometrically embedded into the Gromov--Hausdorff metric class $\GH$. This result is a consequence of local geometry description of the class $\GH$ in a sufficiently small neighborhood of a generic metric space. This description is interesting in itself. The technique of optimal correspondences and their distortions is used.
\end{abstract}

\blfootnote{The work is supported supported be Russian Scientific Foundation, Project~21--11--00355.}

%%%%%%%%%%%%%%%%%%%%%%%%%%%%%%
\section*{Introduction}
%%%%%%%%%%%%%%%%%%%%%%%%%%%%%%
\noindent Comparison of metric spaces is of undoubted interest from both theoretical and practical points of view. One possible approach to this general problem is to use some function of the distance between metric spaces: the less two metric spaces ``similar to each other'', the greater the distance between them. The choice of one or another distance function and a particular class of spaces under consideration depends, generally speaking, on the specifics of the problem. At the moment, apparently, the Gromov--Hausdorff distance is most often used for this purpose.

The history of this distance goes back to works of Felix Hausdorff. In 1914, he defined a non-negative symmetric function on pairs of subsets of a metric space $X$ that is equal to the infimum of reals $r$ such that one subset is contained in the $r$-neighborhood of the other one, and vice-versa~\cite{Hausdorff}. F.~Hausdorff proved that this function satisfies the triangle inequality, and, moreover, it is a metric on the family of all non-empty closed bounded subsets of the space $X$. Later on, D.~Edwards~\cite{Edwards} and independently M.~Gromov~\cite{Gromov} generalized the Hausdorff construction to the case of a pair of arbitrary metric space using  isometric embeddings of the spaces into arbitrary ambient spaces, see below formal definitions. The resulting function is now referred as the \emph{Gromov--Hausdorff distance}. Notice that this function is also symmetric, non-negative and satisfies the triangle inequality. Besides, it always vanishes on a pair of isometrical space, therefore isometrical spaces are usually identified in this contexts. On the other hand, the Gromov--Hausdorff distance could be infinite and also can vanishes on a pair of non-isometric spaces. Nevertheless, if one restricts himself to the family $\cM$ of isometry classes of compact metric spaces, then the Gromov--Hausdorff distance satisfies the metric axioms. The corresponding metric space $\cM$ is called the \emph{Gromov--Hausdorff space}. Geometry of this space turns out to be rather tricky, and it is investigated intensively. It is well-known, that $\cM$ is path-connected, Polish  (i.e., separable and complete), geodesic~\cite{INT}, and also, that $\cM$ is not a proper one and has non non-trivial symmetries~\cite{ITSymm}. A detailed introduction to the geometry of the Gromov--Hausdorff distance can be found in ~\cite[Ch.~7]{BurBurIva} and~\cite{IvaTuzHGH}.

The case of arbitrary metric spaces is also very interesting. Many authors often use some modifications of the Gromov--Hausdorff distance. For example, pointed metric spaces are considered, see~\cite{Jen} or~\cite{Herron}. It is assumed that the marked points are identified by the isometric embeddings. In the present paper we continue to investigate the properties  of the classical Gromov--Hausdorff distance on the metric classes $\GH$ and $\cB$ consisting of representatives of isometry classes of all metric space and of all bounded metric spaces, respectively. Here the term  ``class'' is used in the sense of von~Neumann--Bernays--G\"odel set theory (NGBT). At the proper class $\GH$ it turns out to be possible to define correctly the Gromov--Hausdorff distance and an analogue of the metric topology, see~\cite{BIT} and below.  Also in~\cite{BIT} continuous curves in $\GH$ are defined, and it is shown that the Gromov--Hausdorff distance is an intrinsic generalized semi-metric as on $\GH$, so as on $\cB$, i.e., the distance between any two points equals the infimum of the lengths of the curves connecting these points.

In papers~\cite{BogaTuz} and~\cite{ITsT} the geometry of path connected components of the class $\GH$ is investigated (one of those components turns out to be the class $\cB$), and also the path connectivity of the spheres in $\GH$,  in $\cB$, and in $\cM$ is proved for some specific cases. Also in~\cite{ITsT} a concept of a generic metric space is introduced and some properties of such spaces are described, see below.

In the present paper the geometry of a sufficiently small neighborhood of a generic metric space is investigated. It is shown, see Theorems~\ref{thm:diz} and~\ref{thm:LocIsoRcNtoGHn}, that such neighborhoods are isometric to the balls in the metric space $\R^\cN$ with the metric that equals the supremum of the coordinates differences` absolute values (here $\cN$ stands for an appropriate cardinal number). Basing on these results an isometric embedding of an arbitrary bounded metric space into the Gromov--Hausdorff metric class is constructed, see Theorem~\ref{thm:embed}. Notice that the case of non-bounded metric spaces remains open.

%%%%%%%%%%%%%%%%%%%%%%%%%%%%%%
\section{Preliminaries}
%%%%%%%%%%%%%%%%%%%%%%%%%%%%%%
\noindent Let $X$ be an arbitrary set. By $\#X$ we denote the cardinality of the set $X$, and by $\cP_0(X)$ the set of all its non-empty subsets. A \emph{distance function on the set $X$} is a symmetric mapping $d\:X\x X\to[0,\infty]$ that vanishes at the pairs of the same elements. If $d$ satisfies the triangle inequality, then $d$ is referred as a \emph{generalized semi-metric}. If in addition $d(x,y)>0$ for all $x\ne y$, then $d$ is called a \emph{generalized metric}. At last, if $d(x,y)<\infty$ for all $x,y\in X$, then such function is called a \emph{metric}, and sometimes a \emph{finite metric\/} to underline that $d$ takes finite values only. A set $X$ endowed with a (generalized) (semi-)metric is called a \emph{\(generalized\/\) \(semi-\/\)metric space}.

Below we need the following simple properties of metrics.

\begin{prop}\label{prop:metric}
The following statements are valid.
\begin{enumerate}
\item \label{prop:metric:2} A non-trivial non-negative linear combination of two metrics given at an arbitrary set is a metric at this set.
\item \label{prop:metric:3} A positive linear combination of a metric and a semi-metric given at arbitrary set is a metric at this set.
\end{enumerate}
\end{prop}

If $X$ is a set with some fixed distance function, then the distance between its points $x$ and $y$ is denoted by $|xy|$. If we need to underline that the distance between $x$ and $y$ is calculated in $X$, than we write $|xy|_X$. Further, if $\g\:[a,b]\to X$ is a continuous curve in $X$, then its \emph{length $|\g|$} is defined as the supremum of the ``lengths of inscribed polygonal lines'', i.e., of the values $\sum_i\big|\g(t_i)\g(t_{i+1})\big|$, where the supremum is taken over all finite partitions $a=t_1<\cdots<t_k=b$ of the segment $[a,b]$. A distance function on $X$ is called  \emph{intrinsic}, if the distance between any two points $x$ and $y$ equals to the infimum of the lengths of curves connecting these points.  A curve $\g$ whose length differs by at most $\e$ from $|xy|$ is called \emph{$\e$-shortest}. If for each pair of points $x$ and $y$ of the space $X$ there exists a curve, whose length equals the infimum of the lengths of curves connecting these points and equals $|xy|$, then the distance function is said to be \emph{strictly intrinsic}, and the metric space $X$ is called \emph{geodesic}.

Let $X$ be a metric space. For any $A,\,B\in\cP_0(X)$ and $x\in X$ put
\begin{gather*}
|xA|=|Ax|=\inf\bigl\{|xa|:a\in A\bigr\},\qquad |AB|=\inf\bigl\{|ab|:a\in A,\,b\in B\bigr\},\\
d_H(A,B)=\max\{\sup_{a\in A}|aB|,\,\sup_{b\in B}|Ab|\}=\max\bigl\{\sup_{a\in A}\inf_{b\in B}|ab|,\,\sup_{b\in B}\inf_{a\in A}|ba|\bigr\}.
\end{gather*}
The function $d_H\:\cP_0(X)\x\cP_0(X)\to[0,\infty]$ is called the \emph{Hausdorff distance}. It is well-known, see for example~\cite{BurBurIva} or~\cite{IvaTuzHGH}, that $d_H$ is a metric on the subfamily $\cH(X)\ss\cP_0(X)$ of all closed bounded subsets of $X$.

Let $X$ and $Y$ be metric spaces. A triple $(X',Y',Z)$ consisting of a metric space $Z$ and two its subsets $X'$ and $Y'$ isometric to $X$ and $Y$, respectively, is called a \emph{realization of the pair $(X,Y)$}. The \emph{Gromov--Hausdorff distance $d_{GH}(X,Y)$ between $X$ and $Y$} is defined as the infimum of reals $r$ such that there exists a realization $(X',Y',Z)$ of the pair $(X,Y)$ with $d_H(X',Y')\le r$.

Notice that the Gromov--Hausdorff distance could take as finite, so as infinite values, and always satisfies the triangle inequality, see~\cite{BurBurIva} or~\cite{IvaTuzHGH}. Besides, this distance always vanishes at each pair of isometric spaces, therefore, due to triangle inequality, the Gromov--Hausdorff distance does not depend on the choice of representatives of isometry classes. But there are examples of non-isometric metric spaces with zero Gromov--Hausdorff distance between them, see~\cite{Ghanaat}.

Since each set can be endowed with a metric (for example, one can put all the distances between different points to be equal to $1$), then representatives of isometry classes form a proper class. This class endowed with the Gromov--Hausdorff distance is denoted by $\GH$. Here we use the concept  \emph{class\/} in the sense of von Neumann--Bernays--G\"odel set theory (NBG).

Recall that in NBG all objects (analogues of ordinary sets) are called \emph{classes}. There are two types of classes: \emph{sets\/}  (the classes that can be elements of other classes), and \emph{proper classes\/} (all the remaining classes). The class of all sets is an example of a proper class. Many standard operations such as  intersection, complementation, direct product, mapping, etc., are well-defined for arbitrary classes.

Such concepts as a \emph{distance function, \(generalized\/\) semi-metric\/} and \emph{\(generalized\/\) metric\/} are defined in the standard way for any class, as for a set, so as for a proper class, because the direct products and mappings are defined. But a direct transfer of some other structures, such as topology, leads to contradictions.  Indeed, if we defined a topology for a proper class by analogy with the case of sets, then this class has to be an element of the topology, that is impossible due to the definition of proper classes. In paper~\cite{BIT} the following construction is suggested.

For each class $\cC$  consider a  ``filtration'' by subclasses $\cC_n$, each of which consists of all the elements of $\cC$ of cardinality at most $n$, where $n$ is a cardinal number. Recall that elements of a class are sets, therefore cardinality is defined for them. A class $\cC$ such that all its subclasses $\cC_n$ are sets is said to be  \emph{filtered by sets}. Evidently, if a class $\cC$ is a set, then it is filtered by sets.

Thus, let $\cC$ be a class filtered by sets. When we say that the class $\cC$ satisfies some property, we mean the following: Each set $\cC_n$ satisfies this property. Let us give several examples.
\begin{itemize}
\item Let a distance function on $\cC$ be given. It induces an ``ordinary'' distance function on each set $\cC_n$. Thus, for each $\cC_n$ the standard objects of metric geometry such as open balls are defined. Notice that the open balls in each $\cC_n$  are sets. The latter permits to construct the standard metric topology $\t_n$ on $\cC_n$ taking the open balls as a base of the topology. It is clear that if $n\le m$, then $\cC_n\ss\cC_m$, and the topology $\t_n$ on $\cC_n$ is induced from $\t_m$.
\item More general, a \emph{topology\/} on the class $\cC$ is defined as a family of topologies $\t_n$ on the sets  $\cC_n$ satisfying the following \emph{consistency condition}: If $n\le m$, then $\t_n$ is the topology on $\cC_n$ induced from $\t_m$. A class endowed with such a topology is referred as a \emph{topological class}.
\item The presence of a topology permits to define continuous mappings from a topological space $Z$ to a topological class $\cC$. Notice that according to NBG axioms,  for any mapping $f\:Z\to\cC$ from the set  $Z$ to the class $\cC$, the image $f(Z)$ is a set, all elements of $f(Z)$ are also some sets, and hence, the union  $\cup f(Z)$ is a set of some cardinality $n$. Therefore, each element of $f(Z)$ is of cardinality at most  $n$, and so, $f(Z)\ss\cC_n$. The mapping $f$ is called \emph{continuous}, if $f$ is a continuous mapping from  $Z$ to $\cC_n$. The consistency condition implies that for any $m\ge n$, the mapping $f$ is a continuous mapping from  $Z$ to $\cC_m$, and also for any $k\le n$ such that $f(Z)\ss\cC_k$ the mapping $f$ considered as a mapping from $Z$ to $\cC_k$ is continuous.
\item The above arguments allow to define \emph{continuous curves in a topological class $\cC$}.
\item Let a class $\cC$ be endowed with a distance function and the corresponding topology. We say that the distance function is \emph{interior\/} if it satisfies the triangle inequality, and for any two elements from $\cC$ such that the distance between them is finite, this distance equals the infimum of the lengths of the curves connecting these elements.
\item Let a sequence $\{X_i\}$ of elements from a topological class $\cC$ be given. Since the family $\{X_i\}_{i=1}^\infty$ is the image of the mapping $\N\to\cC$, $i\mapsto X_i$, and $\N$ is a set, then, due to the above arguments, all the family $\{X_i\}$ lies in some $\cC_m$. Thus, the concept of  \emph{convergence of a sequence in a topological class\/}  is defined, namely, the sequence converges if it converges with respect to some topology $\t_m$ such that $\{X_i\}\ss\cC_m$, and hence, with respect to any such topology.
\end{itemize}

Our main examples of topological classes are the classes $\GH$ and $\cB$ defined above. Recall that the class $\GH$ consists of representatives of isometry classes of all metric spaces, and the class $\cB$ consists of representatives of isometry classes of all bounded metric spaces. Notice that $\GH_n$ and $\cB_n$ are sets for any cardinal number $n$, see~\cite{BIT}.

The most studied subset of $\GH$ is the set of all compact metric spaces. It is called a \emph{Gromov--Hausdorff space\/} and often denoted by $\cM$. It is well-known, see~\cite{BurBurIva, IvaTuzHGH, INT}, that  the Gromov--Hausdorff distance is an interior metric on $\cM$, and the metric space $\cM$ is Polish and geodesic. In~\cite{BIT} it is shown that the Gromov--Hausdorff distance is interior as on the class $\GH$, so as on the class $\cB$.

As a rule, to calculate the Gromov--Hausdorff distance between a pair of given metric spaces is rather difficult, and for today the distance is known for a few pairs of spaces, see for example~\cite{GrigIT_Sympl}. The most effective approach for this calculations is based on the next equivalent definition of the Gromov--Hausdorff distance, see details in~\cite{BurBurIva} or~\cite{IvaTuzHGH}. Recall that a \emph{relation\/} between sets $X$ and $Y$ is defined as an arbitrary subset of their direct product $X\x Y$. Thus, $\cP_0(X\x Y)$ is the set of all non-empty relations between $X$ and $Y$.

\begin{dfn}
For any $X,Y\in\GH$ and any $\s\in\cP_0(X\x Y)$, the \emph{distortion $\dis\s$ of the relation $\s$} is defined as the following value:
$$
\dis\s=\sup\Bigl\{\bigl||xx'|-|yy'|\bigr|:(x,y),\,(x',y')\in\s\Bigr\}.
$$
\end{dfn}

Notice that relations can be considered as partially defined multivalued mapping. For a relation $R\ss X\x Y$ and $x\in X$, we put $R(x)=\big\{y\in Y:(x,y)\in R\big\}$ and call $R(x)$ the \emph{image of the element $x$ under the relation $R$}.  Similarly, for $y\in Y$ put $R^{-1}(y)=\big\{x\in X:(x,y)\in R\big\}$ and call $R^{-1}(y)$ the \emph{pre-image of the element $y$ under the relation $R$}. Notice that $R(x)$ and $R^{-1}(y)$ could be empty. The image $R(A)$ and the pre-image $R^{-1}(B)$ of $A\subset X$ and $B\subset Y$, respectively, are defined. A relation $R\ss X\x Y$ between sets $X$ and $Y$ is called a \emph{correspondence}, if $R(X)=Y$ and $R^{-1}(Y)=X$, i.e., if $R(x)$ and $R^{-1}(y)$ are not empty for any $x\in X$ and any $y\in Y$. Notice that the correspondences can be considered as (everywhere defined) multivalued surjective mappings. By $\cR(X,Y)$ we denote the set of all correspondences between $X$ and $Y$. The following result is well-known.

\begin{ass}\label{ass:GH-metri-and-relations}
For any $X,Y\in\GH$, it holds
$$
d_{GH}(X,Y)=\frac12\inf\bigl\{\dis R:R\in\cR(X,Y)\bigr\}.
$$
\end{ass}

We need the following estimates that can be easily proved by means of Assertion~\ref{ass:GH-metri-and-relations}. By $\D_1$ we denote the one-point metric space.

\begin{ass}\label{ass:estim}
For any $X,Y\in\GH$, the following relations are valid\/\rom:
\begin{itemize}
\item $2d_{GH}(\D_1,X)=\diam X$\rom;
\item $2d_{GH}(X,Y)\le\max\{\diam X,\diam Y\}$\rom;
\item If at least one of $X$ and $Y$ is bounded, then $\bigl|\diam X-\diam Y\bigr|\le2d_{GH}(X,Y)$.
\end{itemize}
\end{ass}

For topological spaces $X$ and $Y$, their direct product $X\x Y$ is considered as the topological space endowed with the standard topology of the direct product. Therefore, it makes sense to speak about \emph{closed relations\/} and \emph{closed correspondences}.

A correspondence $R\in\cR(X,Y)$ is called \emph{optimal\/} if $2d_{GH}(X,Y)=\dis R$. The set of all optimal correspondences between $X$ and $Y$ is denoted by $\cR_\opt(X,Y)$.

\begin{ass}[\cite{IvaIliadisTuz, Memoli}]\label{ass:optimal-correspondence-exists}
For any $X,\,Y\in\cM$, there exists as a  closed optimal correspondence, so as a realization $(X',Y',Z)$ of the pair $(X,Y)$ which the Gromov--Hausdorff distance between $X$ and $Y$ is attained at.
\end{ass}

\begin{ass}[\cite{IvaIliadisTuz, Memoli}]
For any $X,\,Y\in\cM$ and each closed optimal correspondence $R\in\cR(X,Y)$, the family $R_t$, $t\in[0,1]$, of compact metric spaces, where $R_0=X$, $R_1=Y$, and the space $R_t$, $t\in(0,1)$, is the set $R$ endowed with the metric
$$
\bigl|(x,y),(x',y')\bigr|_t=(1-t)|xx'|+t\,|yy'|,
$$
is a shortest curve in $\cM$ connecting $X$ and $Y$, and the length of this curve equals $d_{GH}(X,Y)$.
\end{ass}

Let $X$ be a metric space and $\l>0$  a real number. By $\l\,X$ we denote the metric space obtained from $X$ by multiplication of all the distances by $\l$, i.e., $|xy|_{\l X}=\l|xy|_X$ for any $x,\,y\in X$. If the space $X$ is bounded, then for $\l=0$ we put $\l X=\D_1$.

\begin{ass} \label{ass:l1}
For any $X,Y\in\GH$ and any $\l>0$, the equality $d_{GH}(\l X,\l Y)=\l\,d_{GH}(X,Y)$ holds. If $X,Y\in\cB$, then, in addition, the same equality is valid for $\l=0$.
\end{ass}

\begin{ass}\label{ass:l1l2}
Let $X\in\cB$ and $\l_1,\l_2\ge0$. Then $2d_{GH}(\l_1 X,\l_2 X)=|\l_1-\l_2|\,\diam X$.
\end{ass}

\begin{rk}
If $\diam X=\infty$, then Assertion~\ref{ass:l1l2} is not true in general. For example, let $X=\R$. The space  $\l\,\R$ is isometric to $\R$ for any $\l>0$, therefore $d_{GH}(\l_1\R, \l_2\R)=0$ for any $\l_1,\l_2>0$.
\end{rk}

%%%%%%%%%%%%%%%%%%%%%%%%%%%%%%
\section{Metric Cones and Clouds}
%%%%%%%%%%%%%%%%%%%%%%%%%%%%%%
\noindent Let $I$ be an arbitrary set of cardinality $n\ge2$ (not necessary a finite one). By $\cN=I^{(2)}$ we denote the family of all its two-element subsets. An element $\{i,j\}\in\cN$ is denoted by $ij$ or by $ji$. Notice that the cardinality $N$ of the set $\cN$ equals $n$, if and only if either $n$ is finite and equals $3$, or $n$ is infinite.

Let $\R^\cN$ be the set of all real-valued functions on $\cN$ considered as a linear space. For $v\in\R^\cN$ by $v_{ij}=v_{ji}$ we denote the values $v(ij)=v(ji)$, i.e., the \emph{coordinates of the vector $v$}. A vector $v\in\R^\cN$ is said to be  \emph{metrical\/} one or simply a \emph{metric}, if $v_{ij}>0$ for all $ij\in\cN$ (non-negative and non-degenerate), and also $v_{ij}+v_{jk}\ge v_{jk}$ for all $ij,jk,ik\in\cN$ (triangle inequality). The subset $C_n\ss\R^\cN$ consisting of all the metrics is called the \emph{metric cone}. Notice that the set $C_n$ is invariant with respect to multiplication by positive numbers.

Let $X$ be an arbitrary metric space of cardinality $n$ and $\eta\:I\to X$ some bijection. Such bijections are referred as  \emph{enumerations of $X$}, and by $B(X)$ we denote the set of all such bijections. Put $\r_X^\eta(ij)=\bigl|\eta(i)\eta(j)\bigr|$, then, due to properties of metrics, we have: $\r_X^\eta\in C_n$. The function $\r_X^\eta$ is called the \emph{distance vector of the space $X$ corresponding to the enumeration $\eta$}.

Conversely, for a vector $v\in C_n$, a set $Z$ having the same cardinality as $I$ has, and an enumeration $\eta\:I\to Z$ of the set $Z$, the function defined on $Z\x Z$ vanishing at the diagonal $\bigl\{(z,z)\bigr\}_{z\in Z}$ and taking each pair  $(z,z')$, $z\ne z'$ to the number $v_{\eta^{-1}(z)\eta^{-1}(z')}$ is a metric, and  $v=\r_Z^\eta$. Thus, the cone  $C_n$ consists of all distance vectors of all metric spaces of cardinality $n$ constructed by all their possible enumerations.

\begin{rk}
If for a finite $n$ we put $I=\{1,\ldots,n\}$, and put in correspondence to each subset $\{i,j\}\ss I$, $i\ne j$, the basic vector $e_{ij}=e_{ji}$ of the space $\R^N$, then we obtain a natural identification of the spaces $\R^N$ and $\R^\cN$.
\end{rk}

Endow $\R^\cN$ with a distance function $|vw|=\frac12\sup_{ij\in\cN}|v_{ij}-w_{ij}|$ for all $v,w\in\R^\cN$. This distance is a generalized metric. A vector $v\in\R^\cN$ is called \emph{bounded}, if there exists a number $c>0$ such that  $|v_{ij}|<c$ for all $ij\in\cN$. Otherwise, the vector $v$ is called \emph{non-bounded}.

It is easy to see that the property of vectors from $\R^\cN$ to be located at a finite distance, i.e., to have a finite difference, is an equivalence relation. The corresponding equivalence classes we call  \emph{clouds}.  A cloud containing the origin  $0\in\R^\cN$ consists of all bounded vectors and is denoted by $\R^\cN_\infty$. It is clear that $\R^\cN_\infty$ is a linear subspace in $\R^\cN$, and the distance from $0$ in it equals the half of the $\sup$-norm $\|v\|_\infty=\sup_{ij\in\cN}|v_{ij}|$.\footnote{The $\sup$-norm can be defined on the whole space $\R^\cN$. The resulting function could take infinite values. Such functions are called  \emph{generalized norms}. A complication appears with absolute homogeneity axiom $\|\l v\|=|\l|\,\|v\|$, because of indeterminacy appearing for $\l=0$ and $\|v\|=\infty$. One of possible ways to bypass the problem is to put $0\cdot\infty=0$ as in measure theory. Under this approach the mapping $(\l,v)\to\|\l v\|$ is not continuous at $(0,v)$ for $\|v\|=\infty$. Another approach is to claim absolute homogeneity for non-zero $\l$ only.}

Describe the other clouds.

\begin{prop}
Each cloud that is different from $\R^\cN_\infty$ is an affine subspace of the form $v+\R^\cN_\infty$, where $v$ is an non-bounded vector. If $v$ is a non-bounded vector, then the vectors $v$ and $\l v$ lie in different clouds for each $\l\ne 1$.
\end{prop}

\begin{proof}
To prove the first statement notice that the condition
$$
2|vw|=\sup_{ij\in\cN}|v_{ij}-w_{ij}|<\infty
$$
is equivalent to boundedness of the vector with coordinates $v_{ij}-w_{ij}$, i.e., to the condition $v-w\in\R^\cN_\infty$.

Pass to the second statement. Let $v\in\R^\cN$ be a non-bounded vector, then there exists a sequence $v_{i_kj_k}$ tending to infinity. Hence,
$$
2\bigl|(\l v)v\bigr|\ge|\l-1|\cdot|v_{i_kj_k}|\to\infty.
$$
Proposition is proved.
\end{proof}

The metric class $\GH$ can be also partitioned into clouds in the similar manner. The property of metric spaces $X$ and $Y$ to be located at a finite Gromov--Hausdorff distance is an equivalence. The equivalence classes are also called \emph{clouds}, the subclass $\cB$ is an example of such a cloud. In~\cite{ITsT} it is shown that the clouds in $\GH$ are the path-connected components of $\GH$.

Let $\GH_{[n]}\ss\GH$ consists of all metric spaces of cardinality $n$. As it have been already mentioned above, $\GH_n$ is a set, and hence, $\GH_{[n]}\ss\GH_n$ is also a set. As in the case of finite spaces, define the mapping  $\pi\:C_n\to\GH_{[n]}$  putting $\pi(v)$ to be equal to the representative of the isometric class of the metric space $I$ of cardinality $n$, such that $|ij|=v_{ij}$, $i,j\in I$, $i\ne j$, and $|ii|=0$, $i\in I$. Notice that this representative is unique in $\GH_{[n]}$. The mapping $\pi$ is called the  \emph{canonical projection}.

\begin{prop}\label{prop:CanonProj1Lip}
The mapping $\pi\:C_n\to\GH_{[n]}$ is $1$-Lipschitz, i.e.,
$$
d_{GH}\bigl(\pi(v),\pi(w)\bigr)\le|vw|.
$$
\end{prop}

\begin{proof}
If $v,w\in\R^\cN$ belong to a same cloud, then put $X=\pi(v)$, $Y=\pi(w)$, and let $\eta\in B(X)$, $\mu\in B(Y)$ be their enumerations such that $v=\r_X^{\eta}$ and $w=\r_Y^{\mu}$, i.e., $\pi(v)=X$ and $\pi(w)=Y$. Further, put $x_i=\eta(i)$ and $y_i=\mu(i)$ that defines a bijective correspondence $R\in\cR(X,Y)$ as follows\/\rom: $R=\big\{(x_i,y_i)\big\}_{i=1}^n$. So,
$$
2|vw|=\sup_{ij\in\cN}|v_{ij}-w_{ij}|=\sup_{ij\in\cN}\bigl||x_ix_j|-|y_iy_j|\bigr|=\dis R\ge2d_{GH}(X,Y),
$$
that is required.

If $|vw|=\infty$, then the inequality is trivial.
\end{proof}

Let $\cB_{[n]}$ stands for the set of all bounded metric spaces of cardinality $n$.

\begin{prop}
We have\/\rom:  $\pi(C_n\cap\R^\cN_\infty)=\cB_{[n]}$. Moreover, $\pi(C_n\cap\cA)\cap\cB_{[n]}=\0$ for any cloud $\cA\ss\R^\cN$ that is different from $\R^\cN_\infty$.
\end{prop}

\begin{proof}
Boundedness of the metric space $X$ is equivalent to totally boundedness of all the distances in $X$, and hence, of all the components of the vector $\r^\eta_X$ for an arbitrary enumeration $\eta\in B(X)$.
\end{proof}

\begin{rk}
The image of any cloud from $\R^\cN$ under the canonical projection $\pi$ is contained in some cloud in $\GH$. But clouds that are distinct from $\R^\cN_\infty$ can be mapped into the same cloud in $\GH$. Indeed, let $X=\N$ be the set of positive integers with the natural distance. By $\nu$ we denote the bijective mapping of the set  $A=\{2^{n-1}\}_{n\in\N}$ onto the set $B\ss\N$ of odd numbers, namely, $\nu(2^{n-1})=2n-1$. Let $\mu\:\N\sm A\to\N\sm B$ be an arbitrary bijection. By $\eta\:\N\to\N$ we denote the enumeration coinciding with $\nu$ on $A$ and with $\mu$ on $\N\sm A$. Let $\id\:\N\to\N$ be the identical enumeration. Put $v=\r_\N^\id$ and $w=\r_\N^\eta$. Then
$$
|vw|\ge\sup_{n\in\N}\Bigl|\bigl((2n+1)-(2n-1)\bigr)-(2^n-2^{n-1})\Bigr|=\infty,
$$
therefore $v$ and $w$ belong to different clouds of the corresponding space $\R^\cN$, but $\pi(v)=\pi(w)=\N$.
\end{rk}

The following natural question appears. Let two metric spaces $X$ and $Y$ have the same cardinality, and the Gromov--Hausdorff distance between $X$ and $Y$ be finite. Is it true that between $X$ and $Y$ there exists a bijection with finite distortion? In other words, Is it true that there exist enumerations $\eta\in B(X)$ and $\mu\in B(Y)$ such that the distance between the vectors $\r^\eta_X$ and $\r^\mu_Y$ is finite? The following Example demonstrates that it is not true.

\begin{examp}
Let $X=\{2^n\}_{n\in\N}\ss\R$ and $Y=X\cup\{2^n+1\}_{n\in\N}\ss\R$. Consider the correspondence $R\in\cR(X,Y)$ of the form $\bigl\{(2^n,2^n),(2^n,2^n+1)\bigr\}_{n\in\N}$, then $\dis R=1$, therefore $d_{GH}(X,Y)<\infty$. Let $R'\in\cR(X,Y)$ be an arbitrary bijection. Then for any $k\in\N$ there exist  $(2^m,2^n),(2^p,2^n+1)\in R'$ such that $m,n,p>k$. Since $R'$ is bijective, then $p\ne m$, and hence, $\dis R'\ge|2^p-2^m|-1\ge2^{\min(p,m)}-1\to\infty$ as $k\to\infty$. Thus, $\dis R'=\infty$ for any bijection $R'$.
\end{examp}

%%%%%%%%%%%%%%%%%%%%%%%%%%%%%%
\section{Generic Metric Spaces}
%%%%%%%%%%%%%%%%%%%%%%%%%%%%%%
The papers~\cite{ITFiniteLoc} and~\cite{IIT} introduce the notion of a generic finite metric space and describe the local geometry of the Gromov--Hausdorff space $\cM$ in a neighborhood of such a space. Using these ideas, we construct an embedding of an arbitrary finite metric space in $\cM$. In the article~\cite{Filin} the notion of a generic metric space is modified and similar results are obtained.

To transfer the concept of a generic space to the case of infinite spaces we need to introduce some numerical characteristics of metric spaces.

Let $X$ be a metric space, $\#X\ge3$. By $S(X)$ we denote the set of all bijections of the set $X$ onto itself and let $\id\in S(X)$ be the identical bijection. Put:
\begin{align*}
s(X)&=\inf\bigl\{|xx'|:x\ne x',\ x,x'\in X\bigr\},\\
t(X)&=\inf\bigl\{|xx'|+|x'x''|-|xx''|:x\ne x'\ne x''\ne x,\,x,x',x''\in X\bigr\},\\
e(X)&=\inf\bigl\{\dis f:f\in S(X),\ f\ne\id\bigr\}.
\end{align*}

A metric space $M\in\GH$, $\#M\ge3$, is said to be \emph{generic}, if all the three values  $s(M)$, $t(M)$, and $e(M)$ are positive.

\begin{rk}
Positivity of $s(M)$ means that the distances between distinct points of $M$ are separated from zero, in particular, the space $M$ is discrete. A space $M$ such that $s(M)>0$ is said to be \emph{totally discrete}.  The value $t(M)$ characterizes non-degeneracy of triangles from $M$. A space $M$ such that $t(M)>0$ is called \emph{totally non-degenerate}. At last, the value $e(M)$ measures non-symmetry of $M$. A space $M$ such that $e(M)>0$ is called  \emph{totally asymmetric}. Thus, a \emph{generic metric space\/} is a metric space that is totally discrete, totally non-degenerate and totally asymmetric simultaneously.
\end{rk}

\begin{rk}
First attempts to investigate infinite generic spaces were undertaken by A.~Filin in his Diploma Thesis (he proved several Items from Theorem~\ref{thm:diz} and an analogue of Proposition~\ref{prop:optimalRelationPartition}).
\end{rk}

It is not difficult to construct an example of generic metric space $M$ of any finite cardinality $n\ge3$. To do that one can start with the the one-distance metric space $\D_n$ and add to all its non-zero distances different real numbers, whose absolute values are at most $1/3$. Then $s(M)>2/3$, $t(M)>0$, and $e(M)>0$, therefore $M$ is a generic space. Now let us show how to construct a generic metric space of arbitrary infinite cardinality. This construction has been suggested to us by Konstantin Shramov.

\begin{constr}\label{constr:Shramov}
Let $X$ be an arbitrary infinite set. Recall that each set can be totally ordered, i.e., there exists a linear order such that each subset contains the least element with respect to the order. It is easy to verify, see for example~\cite[Ch.~2]{Roman}, that there are no transformations preserving a total order except the identical mapping.

Fix some total order on $X$ and denote it by $<$. Consider the oriented graph $G_o$ with vertex set $X$, whose edges are all ordered pairs $(x,y)$ such that $x<y$. It follows from above, that the graph $G_o$ has unique automorphism, namely, the identical mapping. Now let us construct a new oriented graph $H_o$ changing each edge $e=(x,y)$ of the graph $G_o$ by three vertices $u_e$, $v_e$, $w_e$ and four edges $(x,u_e)$, $(u_e,v_e)$, $(v_e,y)$, $(v_e,w_e)$, see Figure~\ref{fig:shramov}.

\ig{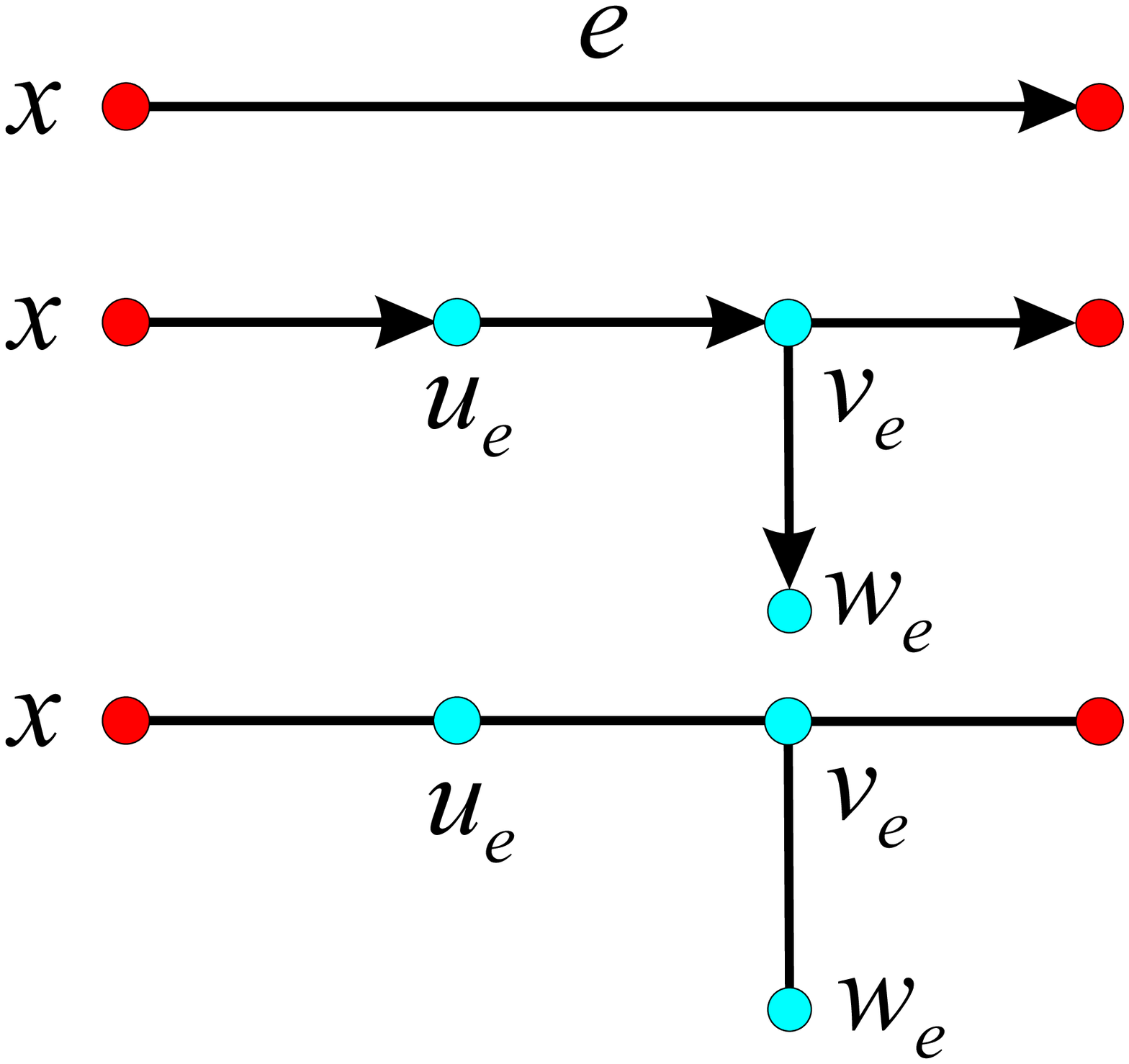}{0.2}{fig:shramov}{Construction of the graph $H$.}

Let $H$ be the non-oriented graph corresponding to $H_o$. Show that the graph $H$ has no non-trivial automorphisms. Notice that the set $X$ coincides with the set of vertices of $H$ having infinite degree, therefore each automorphism $\nu$ of $H$ takes $X$ onto itself. Let $e=(x,y)$ be an edge of the graph $G_o$. We show that $\bigl(\nu(x),\nu(y)\bigr)$ is an edge of the graph $G_0$ also. Assume the contrary, i.e.,  $f=\bigl(\nu(y),\nu(x)\bigr)$ is an edge of the graph $G_0$. The automorphism $\nu$ takes the path $xu_ev_ey$ in the graph $H$ to a path $\nu(x)\nu(u_e)\nu(v_e)\nu(y)$. Since $H$ contains unique path of the length $3$ connecting $\nu(x)$ and $\nu(y)$, and this path is $\nu(y)u_fv_f\nu(x)$, then we conclude that $\nu(v_e)=u_f$ and $\nu(u_e)=v_f$. But the degree of the vertex $v_e$ in the graph $H$ equals $3$, and the degree of the vertex $u_f$ equals $2$, a contradiction. So, $\bigl(\nu(x),\nu(y)\bigr)$ is an edge of the graph $G_o$. Thus, the restriction of each automorphism $\nu$ of the graph $H$ onto $X$ preserves the order, therefore, the restriction of  $\nu$ onto $X$ is the identical mapping. But then the vertices of the paths $xu_ev_ey$ are also fixed, and so, $w_e$ goes to itself also.

Let $V$ be the vertex set of the graph $H$. Notice that in the case of an infinite $X$ the sets $V$ and $X$ have the same cardinality. Define a metric on $V$ as follows: put the distance between non-adjacent vertices to be equal to  $1$, and the distance between adjacent vertices to be equal to $1+\e$, where $0<\e<1$. It is clear that $s(V)=1$, and $t(V)=1-\e>0$. Notice that the distortion of each bijection of the set $V$ onto itself that is not an automorphism of the graph $H$ equals $\e$, therefore, $e(V)=\e$ because $H$ has no non-trivial automorphisms. Thus, $V$ is  generic.
\end{constr}

Next Theorem contains several properties of totally discrete spaces.

\begin{thm}\label{thm:diz}
Let $M\in\GH$ be a totally discrete metric space, i.e., $\#M\ge3$ and $s(M)>0$. Fix an arbitrary $\e\in\bigl(0,s(M)/2\bigr]$. Then for any space $X\in\GH$ such that $r:=d_{GH}(M,X)<\e$ the following statements hold.
\begin{enumerate}
\item\label{thm:diz:1} There exists a correspondence $R\in\cR(M,X)$ such that $\dis R<2\e$.
\item\label{thm:diz:2} For any such a correspondence $R$, if $i,j\in M$, $i\ne j$, then $R(i)\cap R(j)=\0$, and so the family  $D_R=\bigl\{X_i:=R(i)\bigr\}_{i\in M}$ is a partition of the space $X$.
\item\label{thm:diz:3} For any $i,j\in M$, $i\ne j$ and arbitrary $x_i\in X_i$ and $x_j\in X_j$ the inequality $\bigl||x_ix_j|-|ij|\bigr|<2\e$ is valid.
\item\label{thm:diz:4} For all $i\in M$ the inequality $\diam X_i<2\e$ holds.
\item\label{thm:diz:5} If $\e\le s(M)/4$, then the partition $D_R$ is defined uniquely, i.e., if $R'\in\cR(M,X)$ is another partition such that  $\dis R'<2\e$, then $D_{R'}=D_R$.
\item\label{thm:diz:6} If in addition $M$ is totally asymmetric, i.e., $e(M)>0$, and besides, $\e\le s(M)/4$ and $\e<e(M)/4$, then the correspondence $R\in\cR(M,X)$ such that $\dis R<2\e$ is defined uniquely, and hence,  $R$ is an optimal correspondence.
\end{enumerate}
\end{thm}

\begin{proof}
(\ref{thm:diz:1}) Since
$$
d_{GH}(M,X)=\frac12\inf\bigl\{\dis R:R\in\cR(M,X)\bigr\}<\e,
$$
then there exists a correspondence $R\in\cR(M,X)$ such that $\dis R<2\e$.

(\ref{thm:diz:2}) If there exists a point $x\in X$ such that $x\in R(i)\cap R(j)$, $i\ne j$, then $\#R^{-1}(x)>1$, and so,
$$
\dis R\ge\diam R^{-1}(x)\ge s(M)\ge2\e,
$$
a contradiction. Therefore, different $X_i:=R(i)$ do not intersect each other and form a partition $\{X_i\}_{i\in M}$ of the space $X$.

(\ref{thm:diz:3}) Since $\dis R<2\e$, then for any $i,j\in M$, $i\ne j$, and arbitrary $x_i\in R(i)$, $x_j\in R(j)$ the inequality
$$
\bigl||x_ix_j|-|ij|\bigr|\le\dis R<2\e
$$
holds.

(\ref{thm:diz:4}) Since $\dis R<2\e$, then for each $i\in M$ the inequality $\diam X_i\le\dis R<2\e$ holds.

(\ref{thm:diz:5}) If $\e\le s(M)/4$ and $R'\in\cR(M,X)$ is another correspondence such that $\dis R'<2\e$, then $\big\{X'_i:=R'(i)\big\}_{i\in M}$ is also a partition of $X$ satisfying the properties listed above. If $D_R\ne D_{R'}$, then there exists either some $X'_i$ intersecting some different $X_j$ and $X_k$ simultaneously, or some $X_i$ intersecting some different $X'_j$ and $X'_k$ simultaneously. Indeed, if there is no such an $X'_i$, then each $X'_i$ is contained in some $X_j$, and hence, $\{X'_p\}$ is a sub-partition of $\{X_q\}$. Since the partitions $D_R$ and $D_{R'}$ are supposed to be different, then some element $X_i$ contains several different $X'_j$.

Thus, without loss of generality assume that $X'_i$ intersects some distinct  $X_j$ and $X_k$ simultaneously. Choose arbitrary $x_j\in X_j\cap X'_i$ and $x_k\in X_k\cap X'_i$. Then
$$
2\e>\diam X'_i\ge|x_jx_k|>|jk|-2\e\ge s(M)-2\e,
$$
and so $\e>s(M)/4$, a contradiction.

(\ref{thm:diz:6}) As in the previous Item, assume that $R'\in\cR(M,X)$ satisfies the inequality $\dis R'<2\e$, and hence, as it has been already proved, $D_R=D_{R'}$. It remains to verify that $X_i=X'_i$ for all $i\in M$. Assume the contrary, i.e., there exists a non-trivial bijection $\s\:M\to M$ such that $X_i=X'_{\s(i)}$. Due to assumptions, $\dis\s\ge e(M)>4\e$. The latter means that there exist $i,j\in M$ such that
$$
\Bigl||ij|-\bigl|\s(i)\s(j)\bigr|\Bigr|>4\e.
$$
Put $p=\s(i)$, $q=\s(j)$ and choose arbitrary $x_i\in X_i=X'_p$ and $x_j\in X_j=X'_q$. Then
$$
4\e<\bigl||pq|-|ij|\bigr|=\bigl||pq|-|x_ix_j|+|x_ix_j|-|ij|\bigr|\le
\bigl||pq|-|x_ix_j|\bigr|+\bigl||x_ix_j|-|ij|\bigr|<2\e+2\e,
$$
a contradiction. Thus, it is proved that the correspondence $R$ whose distortion is sufficiently close to $2d_{GH}(M,X)$ is uniquely defined. The latter implies that $R$ is optimal.
\end{proof}

\begin{dfn}
The family $\{X_i\}$ from Theorem~\ref{thm:diz} we call the \emph{canonical partition of the space $X$ with respect to  $M$}.
\end{dfn}

\begin{prop}\label{prop:optimalRelationPartition}
Let $M$ be a totally discrete space, i.e., $\#M\ge3$ and $s(M)>0$. Fix an arbitrary $\e\in\bigl(0,s(M)/8\bigr]$. Then for any  $X,Y\in\GH$ such that $d_{GH}(M,X)<\e$ and $d_{GH}(M,Y)<\e$ the corresponding canonical partitions  $\{X_i\}_{i\in M}$ and $\{Y_i\}_{i\in M}$ with respect to $M$ are uniquely defined and there exists a correspondence $R\in\cR(X,Y)$ such that $\dis R<4\e$. For each such correspondence there exists a bijection $\s\:M\to M$ such that $R$ can be represented in the form $R=\sqcup_{i\in M}R_i$ for some $R_i\in\cR(X_i,Y_{\s(i)})$.
\end{prop}

\begin{proof}
Due to the triangle inequality, we have  $d_{GH}(X,Y)\le d_{GH}(X,M)+d_{GH}(M,Y)<2\e$ that implies the existence of $R$.

Further, choose arbitrary $x,x'\in X_i$ and $y\in R(x)$, $y'\in R(x')$, and show that $y$ and $y'$ belong to the same $Y_j$. Assume the contrary, i.e.,  $y\in Y_j$ and $y'\in Y_k$ for some $j\ne k$. Then Theorem~\ref{thm:diz} implies that
$$
|xx'|\le\diam X_i<2\e,\qquad  |yy'|>|jk|-2\e\ge s(M)-2\e\ge6\e,
$$
and hence, $\dis R\ge\bigl||xx'|-|yy'|\bigr|>4\e$, a contradiction.

Swapping $X$ and $Y$ conclude that for any $(x,y),\,(x',y')\in R$, the points $x$ and $x'$ belong to the same element of the canonical partition $\{X_i\}$, if and only if the points $y$ and $y'$ belong to the same element of the canonical partition $\{Y_i\}$ that completes the proof.
\end{proof}

\begin{prop}\label{prop:correspondence-partition}
Let $M\in\GH$ be a totally discrete and totally asymmetric space, i.e., $\#M\ge3$, $s(M)>0$, and $e(M)>0$. Fix an arbitrary  $\e\in\bigl(0,s(M)/8\bigr]$ such that $\e<e(M)/8$. Then for any $X,Y\in\GH$ such that $d_{GH}(M,X)<\e$, $d_{GH}(M,Y)<\e$, the canonical partitions $\{X_i\}_{i\in M}$ and $\{Y_i\}_{i\in M}$ with respect to $M$ are uniquely defined and there exists a correspondence $R\in\cR(X,Y)$ such that $\dis R<4\e$. Moreover, each such correspondence $R$ can be represented in the form $R=\sqcup_{i\in M}R_i$ for some $R_i\in\cR(X_i,Y_i)$.
\end{prop}

\begin{proof}
Proposition~\ref{prop:optimalRelationPartition} implies that there exists a correspondence $R\in\cR(X,Y)$ such that $\dis R<4\e$, and that for each such a correspondence there exists a bijection $\s\:M\to M$ and correspondences $R_i\in\cR(X_i,Y_{\s(i)})$ such that $R=\sqcup_{i\in M}R_i$. It remains to show that $\s$ is identical.

Assume the contrary, then $\dis\s\ge e(M)>8\e$, and hence, there exist $i,j\in M$ such that  for $k=\s(i)$ and $l=\s(j)$ the inequality $\bigl||ij|-|kl|\bigr|>8\e$ holds. But due to the assumptions, for any $x_i\in X_i$, $x_j\in X_j$,  $y_k\in R(x_i)\ss Y_k$, and $y_l\in R(x_j)\ss Y_l$ the inequalities $\bigl||y_ky_l|-|kl|\bigr|<2\e$, $\bigl||x_ix_j|-|ij|\bigr|<2\e$, and $\bigl||x_ix_j|-|y_ky_l|\bigr|\le\dis R<4\e$ are valid. So, we get the following estimate:
\begin{multline*}
\bigl||ij|-|kl|\bigr|=\bigl||ij|-|x_ix_j|+|x_ix_j|-|y_ky_l|+|y_ky_l|-|kl|\bigr|\le \\ \le \bigl||ij|-|x_ix_j|\bigr|+\bigl||x_ix_j|-|y_ky_l|\bigr|+\bigl||y_ky_l|-|kl|\bigr|<8\e,
\end{multline*}
a contradiction.
\end{proof}

\begin{thm}\label{thm:LocIsoRcNtoGHn}
Let $n\ge3$ be a cardinal number, $I$ a set of cardinality $n$, $\cN=I^{(2)}$ the set of all two-element subsets of  $I$, $C_n\ss\R^\cN$ the metric cone, and $\pi\:C_n\to\GH_{[n]}$ the canonical projection. Assume that $\GH_{[n]}$ contains a totally discrete and totally asymmetric space $M$, and let $w\in\pi^{-1}(M)$. Fix some $\e\in\bigl(0,s(M)/8\bigr]$ such that $\e<e(M)/8$, and let $U_\e(w)\ss\R^\cN$ be the corresponding open ball in $\R^\cN$. Then the restriction of the canonical projection $\pi$ onto $U_\e:=U_\e(w)\cap C_n$ is isometric.
\end{thm}

\begin{proof}
Let $w=\r^\eta_M$ for some enumeration $\eta\in B(M)$. For $i\in I$ put $m_i=\eta(i)$.

Choose arbitrary $a,b\in U_\e$, and put $X=\pi(a)$, $Y=\pi(b)$. Since the mapping $\pi$ is $1$-Lipschitz in accordance with Proposition~\ref{prop:CanonProj1Lip}, then $d_{GH}(X,M)\le|aw|<\e$ and $d_{GH}(Y,M)\le|bw|<\e$, and hence, due to Theorem~\ref{thm:diz}, for the spaces $X$ and $Y$ corresponding canonical partitions $\{X_i\}_{m_i\in M}$ and $\{Y_i\}_{m_i\in M}$ are uniquely defined. Let $R_X\in\cR(M,X)$ and $R_Y\in\cR(M,Y)$ be correspondences satisfying $\dis R_X<2\e$, $\dis R_Y<2\e$ and defining these partitions. Due to Theorem~\ref{thm:diz},  $R_X$ and $R_Y$ are uniquely defined and optimal, i.e., $2d_{GH}(M,X)=\dis R_X$ and $2d_{GH}(M,Y)=\dis R_Y$.

For $x\in\R^\cN$ put $s(x)=\inf_{ij\in \cN} x_{ij}$. Notice that $s(X)=s(a)\ge s(w)-2\e=s(M)-2\e\ge6\e$, and similarly $s(Y)\ge6\e$. On the other hand, $\diam X_i<2\e$ and $\diam Y_i<2\e$ for all  $m_i\in M$ in accordance with Theorem~\ref{thm:diz}, therefore all the $X_i$ and $Y_i$ are one-point sets, i.e., $X_i=\{x_i\}$ and $Y_i=\{y_i\}$. Thus, $R_X=\{(m_i,x_i)\}_{i\in I}$ and $R_Y=\{(m_i,y_i)\}_{i\in I}$.

Due to Proposition~\ref{prop:correspondence-partition}, there exists a correspondence $R\in\cR(X,Y)$ such that $\dis R<4\e$, and each such correspondence has the form  $R=\sqcup_{m_i\in M}R_i$, where $R_i\in\cR(X_i,Y_i)$. Since  $X_i$ and $Y_i$ are one-point sets, then $R$ is a bijection consisting of all pairs $(x_i,y_i)$. Thus, $R$ is uniquely defined. Therefore, if the distortion of a correspondence from $\cR(X,Y)$ is less than $4\e$, then it coincides with $R$, and so, $R$ is an optimal correspondence, and hence, $2d_{GH}(X,Y)=\dis R$.

Notice that the mappings $\eta_X\:i\mapsto x_i$ and $\eta_Y\:i\mapsto y_i$ are enumerations of $X$ and $Y$. Therefore, the vectors $\r_X:=\r^{\eta_X}_X$ and $\r_Y:=\r^{\eta_Y}_Y$ from the metric cone $C_n$ are defined. Show that $a=\r_X$ and $b=\r_Y$. To do that, notice that $|aw|<\e$ and on the other hand
$$
|\r_Xw|=\frac12\sup_{ij\in\cN}\bigl||x_ix_j|-|m_im_j|\bigr|=\frac12\dis R_X=d_{GH}(M,X)<\e.
$$
If $\eta'_X$ is an enumeration of $X$ distinct from $\eta$, and $\r'_X:=\r^{\eta'_X}_X$, then $|\r_X\r'_X|=\frac12\dis\s_X$ for the non-identical bijection $\s_X\:X\to X$ corresponding to the renumeration, i.e., $\s_X=\eta'_X\c\eta^{-1}_X$. Let $\s_X(x_i)=x_{i'}$. Then the bijection $\s_M\:M\to M$, $\s_M\:m_i\mapsto m_{i'}$ is defined. Due to assumptions,  $\dis\s_M\ge e(M)>8\e$. But then we have:
\begin{multline*}
|\r'_Xw|=\frac12\sup_{ij\in\cN}\bigl||x_{i'}x_{j'}|-|m_im_j|\bigr|=\\
=\frac12\sup_{ij\in\cN}\bigl||x_{i'}x_{j'}|-|m_{i'}m_{j'}|+|m_{i'}m_{j'}|-|m_im_j|\bigr|\ge
\frac12\sup_{ij\in\cN}\bigl||m_{i'}m_{j'}|-|m_im_j|\bigr|-\\
-\frac12\sup_{ij\in\cN}\bigl||x_{i'}x_{j'}|-|m_{i'}m_{j'}|\bigr|=
\frac12(\dis\s_M-\dis R_X)>\frac12(8\e-2\e)=3\e.
\end{multline*}
So, each point from $\pi^{-1}(X)$ distinct from $\r_X$ lies outside $U_\e$. Similarly for $Y$.

Thus, it is shown that only the points $\r_X$ and $\r_Y$ from $\pi^{-1}(X)$ and $\pi^{-1}(Y)$ lie in $U_\e$, and so, $a=\r_X$ and $b=\r_Y$. It remains to notice that
$$
|ab|=|\r_X\r_Y|=\frac12\sup_{ij\in\cN}\bigl||x_ix_j|-|y_iy_j|\bigr|=\frac12\dis R=d_{GH}(X,Y),
$$
that is required.
\end{proof}

Notice that in the case of infinite cardinality $n$ a space $Z\in\GH_{[n]}$ that is close to a totally discrete and totally asymmetric space $M$ need not be totally discrete (changing points to small subspaces of the same cardinality does not change neither cardinality, nor the distance between $Z$ and $M$). Therefore, the image of the neighborhood $U_\e(w)\cap C_n$ from Theorem~\ref{thm:LocIsoRcNtoGHn} under the canonical projection does not equal to the open ball $U_\e(M)\cap\GH_{[n]}$. In the case of finite $n$ each such space $Z\in\cM_{[n]}$ consists of $n$ points, and the distance $d_{GH}(M,Z)$ attains at some bijection. This bijection in fact transfers enumeration from $M$ to $Z$, and hence, the corresponding vector $\r_Z$ lies in $U_\e(w)\cap C_n$ and so, the canonical projection  $\pi$ is an isometrical mapping from  $U_\e(w)\cap C_n$ onto $U_\e(M)\cap\cM_{[n]}$.

\begin{cor}\label{cor:LocIsoRNtoMn}
Let $n\ge3$ be an integer, $N=n(n-1)/2$, $C\ss\R^N$ the metric cone, and $\pi\:C\to\cM_{[n]}$ the canonical projection. Let  $M\in\cM_{[n]}$ be a totally discrete and totally asymmetric space, and $w\in\pi^{-1}(M)$. Fix an arbitrary $\e\in\bigl(0,s(M)/8\bigr]$ such that $\e<e(M)/8$, and let $U_\e(w)\ss\R^N$ be the corresponding open ball in $\R^N$. Then the restriction of the canonical projection $\pi$ onto $U_\e(w)\cap C$ is isometric. Moreover, if $M$ is generic, i.e., if in addition $t(M)>0$, and $\e\le t(M)/6$, then $U_\e(w)$ lies in the interior of  the cone $C$.
\end{cor}

\begin{proof}
All the statements of Corollary except the last one follows from Theorem~\ref{thm:LocIsoRcNtoGHn} and reasonings given after its proof. Let us prove the last statement.

We need to show that each $v\in U_\e(w)$ lies in the interior of the cone $C$, i.e., all the coordinates of the vector $v$ are positive and all the triangle inequalities are strict. Since $|vw|<\e$, then for each $1\le i<j\le n$ the inequality $|v_{ij}-w_{ij}|<2\e$ holds. And since $\e\le s(M)/8$, then for any pairwise distinct $i,j\in\{1,\ldots,n\}$ we have:
$$
v_{ij}>w_{ij}-2\e\ge s(M)-2\e\ge8\e-2\e=6\e>0,
$$
and hence, all the coordinates of the vector $v$ are positive.

Further, for any pairwise distinct $i,j,k\in\{1,\ldots,n\}$ we have:
\begin{multline*}
v_{ij}+v_{jk}-v_{ik}>w_{ij}-2\e+w_{jk}-2\e-(w_{ik}+2\e)=w_{ij}+w_{jk}-w_{ik}-6\e\ge\\
\ge t(M)-6\e\ge6\e-6\e=0,
\end{multline*}
and hence, all the triangle inequalities in $v$ are strict.
\end{proof}

%%%%%%%%%%%%%%%%%%%%%%%%%%%%%%
\section{Isometric Embedding of a Bounded Metric Space into the Metric Class $\GH$}
%%%%%%%%%%%%%%%%%%%%%%%%%%%%%%
\noindent As a corollary of Theorem~\ref{thm:LocIsoRcNtoGHn} we prove universality of the metric class $\GH$ for bounded metric spaces.

\begin{thm}\label{thm:embed}
Let $X$ be a bounded metric space of cardinality $n$. Then $X$ can be isometrically embedded into $\GH_n$.
\end{thm}

\begin{proof}
For finite $n$ this Theorem is proved in~\cite{IIT}. Therefore we assume that $n$ is an infinite cardinal number. Recall, that in this case $\cN=n$.

Fix some enumeration $\eta$ of the space $X$, put $x_i=\eta(i)$, and define a mapping $f\:X\to\R^n$ as follows: $f(x_i)_j=|x_ix_j|$.

\begin{lem}\label{lem:Kurat}
Define a generalized metric on $\R^n$ as follows\/\rom: $|xy|=\sup_i|x_i-y_i|$. Then the mapping $f$ is isometrical.
\end{lem}

\begin{proof}
Indeed, on one hand, $\big||x_px_j|-|x_qx_j|\big|\le|x_px_q|$ for any $i$, and so,
$$
\big|f(x_p)f(x_q)\big|=\sup_j\big||x_px_j|-|x_qx_j|\big|\le|x_px_q|,
$$
and on the other hand,  for $j=p$ and for $j=q$ the equality attains. Lemma is proved.
\end{proof}

The mapping $f$ is called the Kuratowski embedding~\cite{CKur}, see also~\cite{ITMinFil} and~\cite{IIT} for other examples of similar applications of the Kuratowski embedding. Since $n$ is infinite, we can assume that the Kuratowski embedding takes $X$ to $\R^\cN$.

We denote by $d$ the diameter of the space $X$. Consider a totally discrete and totally asymmetric space $M$ such that $s(M)>8d$ and $e(M)>8d$, and let $w\in\pi^{-1}(M)$. We set $Z=w+f(X)\ss\R^\cN$. Since translations preserve the distance in $\R^\cN$, the space $Z\ss\R^\cN$ is isometric to $X$. We fix some $\e>d$ such that $\e<s(M)/8$ and $\e<e(M)/8$. Then $Z\ss U_\e:=U_\e(w)\cap C_n$. It remains to apply the canonical projection $\pi\:U_\e\to\GH_{[n]}$, which is isometric by Theorem~\ref{thm:LocIsoRcNtoGHn}. The theorem has been proven.
\end{proof}

%%%%%


\begin{thebibliography}{99}
%%%%%

\bibitem{Hausdorff}
F.~Hausdorff, \emph{Grundz\"uge der Mengenlehre}, Leipzig, Veit, 1914.

\bibitem{Edwards}
D.~Edwards, ``The Structure of Superspace'',
in: \emph{Studies in Topology}, ed. by N.\,M.~Stavrakas and K.\,R.~Allen,
New York, London, San Francisco, Academic Press, Inc., 1975.

\bibitem{Gromov}
M.~Gromov, ``Groups of Polynomial growth and Expanding Maps''
in: {\em Publications Mathematiques I.H.E.S.}, vol.~53, 1981.

\bibitem{INT} A.\,O.~Ivanov, N.\,K.~Nikolaeva, A.\,A.~Tuzhilin, ``The Gromov--Hausdorff Metric on the Space of Compact Metric Spaces is Strictly Intrinsic'', ArXiv e-prints, {\tt arXiv:1504.03830}, 2015; Mathematical Notes,  {\bf 100} (6), pp.~171--173, 2016.

\bibitem{ITSymm}
A.\,O.~Ivanov and A.\,A.~Tuzhilin, ``Isometry group of Gromov--Hausdorff space''
\emph{Matematicki Vesnik}, {\bf 71}: (1--2),  pp.~123--154, 2019.

\bibitem{BurBurIva}
D.\,Yu.~Burago, Yu.\,D.~Burago, S.\,V.~Ivanov,  \emph{A Course in Metric Geometry}, Providence, RI, American Mathematical Soc., 2001.

\bibitem{IvaTuzHGH} A.\,O.~Ivanov, A.\,A.~Tuzhilin,  \emph{Geometry of Hausdorff and Gromov--Hausdorff Distances: the Case of Compact Spaces}, Moscow, Izd-vo Popechitel'skogo Soveta Mech. Mat. Fac. MGU, 2017 [in Russian], ISBN: 978-5-9500628-1-0.

\bibitem{Jen} D.~Jansen, ``Notes on Pointed Gromov-Hausdorff Convergence'', {\tt ArXiv:math/1703.09595v1}, 2017.

\bibitem{Herron}
D.\,A.~Herron, ``Gromov–Hausdorff Distance for Pointed Metric Spaces'', \emph{J. Anal.}, {\bf 24} (1), 2016, pp.~1--38.

\bibitem{Ghanaat} P.~Ghanaat, ``Gromov-Hausdorff distance and applications'', In: \emph{Summer school ``Metric Geometry''}, Les Diablerets, August 25--30, 2013,  {\tt https://math.cuso.ch/fileadmin/math/ document/gromov-hausdorff.pdf}

\bibitem{BIT} S.\,I.~Borzov, A.\,O.~Ivanov, A.\,A.~Tuzhilin, ``Extendability of Metric Segments in Gromov--Hausdorff Distance'',  ArXiv e-prints,	 {\tt arXiv:2009.00458}, 2020 (to appear in Matem. Sbornik, 2022).

\bibitem{ITsT} A.~Ivanov, R.~Tsvetnikov, A.~Tuzhilin, ``Path Connectivity of Spheres in the Gromov--Hausdorff Class'',
ArXiv e-prints, {\tt arXiv:2111.06709} (to appear in Topology and Applications, 2022).

\bibitem{BogaTuz} S.\,A.~Bogaty, A.\,A.~Tuzhilin,  ``Gromov--Hausdorff Class: Its Completeness and Cloud Geometry'',  ArXiv e-prints, {\tt arXiv:2110.06101}, 2021.

\bibitem{GrigIT_Sympl} D.\,S.~Grigorjev, A.\,O.~Ivanov, A.\,A.~Tuzhilin, ``Gromov--Hausdorff Distance to Simplexes'', ArXiv e-prints, {\tt arXiv:1906.09644}, 2019; Chebyshevskii Sbornik, {\bf 20} (2), pp.~100--114, 2019.

\bibitem{IvaIliadisTuz} A.\,O.~Ivanov, S.~Iliadis, A.\,A.~Tuzhilin, ``Realizations of Gromov-Hausdorff Distance'', ArXiv e-prints, {\tt arXiv:1603.08850}, 2016.

\bibitem{Memoli} S.~Chowdhury, F.~Memoli,  ''Constructing Geodesics on the Space of Compact Metric Spaces'', ArXiv e-prints, {\tt arXiv:1603.02385}, 2016.

\bibitem{ITFiniteLoc} A.\,O.~Ivanov, A.\,A.~Tuzhilin, ``Local Structure of Gromov-Hausdorff Space near Finite Metric Spaces in General Position'',    ArXiv e-prints, {\tt arXiv:1611.04484}, 2016; Lobachevskii Journal of Mathematics, {\bf 38} (6), pp.~998–1006, 2017.

\bibitem{Filin} A.\,M.~Filin,  ``Local Geometry of the Gromov--Hausdorff Space and Totally Asymmetric Finite Metric Spaces'', Fund. and Prikl. Mat.,  {\bf 22} (6), pp.~263--272, 2019 [English translation: to appear in J. of Math. Sci.].

\bibitem{CKur} C.~Kuratowski,   \emph{Quelques probl\`emes concernant les espaces m\'etriques non-separables}, Fundamenta Math., {\bf 25}, pp.~534--545, 1935.

\bibitem{IIT} S.~Iliadis, A.~Ivanov, A.~Tuzhilin, ``Local structure of Gromov-Hausdorff space, and isometric embeddings of finite metric spaces into this space'', Topology and its Applications, {\bf 221}, pp.~393--398, 2017.

\bibitem{ITMinFil} A.\,O.~Ivanov, A.\,A.~Tuzhilin, “One-Dimensional Gromov Minimal Filling Problem”, Sb. Math., {\bf 203} (5),  pp.~677--726, 2012.

\bibitem{Roman} S.~Roman, \emph{Lattices and Ordered Sets}, Springer, New York, NY, 2008.
\end{thebibliography}
\end{document}